\documentclass[11pt,letterpaper]{article}
\usepackage[margin=1in]{geometry}
\usepackage{amsmath,amssymb,amsthm,mathtools}
\usepackage{microtype}
\usepackage{enumitem}
\usepackage{tikz}
\usetikzlibrary{calc,arrows.meta,decorations.pathreplacing,positioning,angles,quotes}
\usepackage{caption}
\usepackage{placeins}
\usepackage{hyperref}
\hypersetup{
  colorlinks=true,
  linkcolor=blue,
  citecolor=blue,
  urlcolor=blue,
  pdftitle={Reciprocal-Polar Linearization and Two Conic Constructions for the Cone Projection},
  pdfauthor={George M. Georgiou},
  pdfsubject={Euclidean conic geometry and the cone projection},
  pdfkeywords={conic sections, cone projection, reciprocal polarity,
    focus-directrix conics, envelope of normals}
}

\numberwithin{equation}{section}

\theoremstyle{plain}
\newtheorem{theorem}{Theorem}[section]
\newtheorem{proposition}[theorem]{Proposition}
\newtheorem{corollary}[theorem]{Corollary}

\theoremstyle{definition}
\newtheorem{definition}[theorem]{Definition}
\theoremstyle{remark}

\newcommand{\DR}{D_R}
\newcommand{\RR}{\mathbb{R}}
\newcommand{\norm}[1]{\left\lVert #1\right\rVert}

\newcommand{\RP}{\mathcal{R}_{\Sigma_R}}

\title{Reciprocal-Polar Linearization and Two Conic Constructions\\
for the Cone Projection}
\author{\small George M. Georgiou\thanks{\small School of Computer Science and Engineering, California State University, San Bernardino}\\[3pt]
\small \href{mailto:georgiou@csusb.edu}{georgiou@csusb.edu}}
\date{}

\begin{document}
\maketitle

\begin{abstract}
\noindent
Let
\[
  f_R(x)=\frac{x}{1+\norm{x}/R},\qquad R>0,
\]
be the radial cone projection of the plane onto the open disk
$\DR=\{x:\norm{x}<R\}$.  Previous work established the Self-Directrix and
Confocal-Codirectrix Theorems, according to which $f_R$ maps focal conic arcs
to focal conic arcs while preserving the distinguished focus and directrix.
This note combines those results with three classical facts recorded by
W.~H.~Besant.  First, reciprocal polarity with respect to
$\partial\DR$ represents every focal conic under consideration as the reciprocal
polar of a unique circle, and in this circle representation the nonlinear action
of $f_R$ becomes the elementary radius translation $q\mapsto q+R$.  Second, the normals at the
intersections of a fixed ray with the self-directrix family envelope an
explicit parabola.  Third, the tangent at $f_R(z)$ is constructed directly
from the original point $z$ and the original line, without differentiating
$f_R$ or solving for the conic.
\end{abstract}


\section{Introduction and setup}

Write $O$ for the origin and
\[
  \Sigma_R=\partial\DR=\{x\in\RR^2:\norm{x}=R\}.
\]
The radial identity
\begin{equation}\label{eq:lens}
  \frac{1}{\norm{f_R(x)}}=\frac{1}{\norm{x}}+\frac{1}{R}
  \qquad (x\ne O)
\end{equation}
and the fact that $f_R$ preserves every ray from $O$ are the two elementary
features underlying the three constructions below.  By the
\emph{carrier} of an image arc we mean the complete conic containing that arc.

The classical results of Besant serve here as geometric input, but the
statements below concern the specific action of the cone projection.
Theorem~\ref{thm:rp-translation} gives a reciprocal-polar linearization of the
Confocal-Codirectrix law, while Theorem~\ref{thm:normal-envelope} determines
explicitly the parabola enveloped by the normals of the self-directrix sweep.
Proposition~\ref{prop:tangent-construction} specializes a classical
focus-directrix tangent theorem to give a direct construction from the
original input line.

A focal conic with focus $O$, axis direction $\alpha$, eccentricity $e>0$,
and semi-latus rectum $\lambda>0$ has polar equation
\begin{equation}\label{eq:focal-polar}
  \rho(\theta)=\frac{\lambda}{1+e\cos(\theta-\alpha)}
\end{equation}
over the directions for which the denominator is positive.  Its associated
directrix is at distance
\[
  \delta=\frac{\lambda}{e}
\]
from $O$.  The Confocal-Codirectrix parameter law~\cite{Georgiou-cone} is
\begin{equation}\label{eq:cc-law}
  \frac{1}{\lambda'}=\frac{1}{\lambda}+\frac{1}{R},
  \qquad
  \frac{1}{e'}=\frac{1}{e}+\frac{\delta}{R},
  \qquad
  \frac{\lambda'}{e'}=\delta.
\end{equation}
Thus $f_R$ preserves the focus and directrix and changes only the member of
the fixed-directrix pencil; these identities are re-derived transparently in
the next section.  For the limiting line member
$\ell=\{x\cdot n=d\}$, $\norm n=1$, the Self-Directrix Theorem~\cite{Georgiou-cone}
gives the carrier conic
\begin{equation}\label{eq:self-directrix-data}
  \text{focus }O,
  \qquad \text{directrix }\ell,
  \qquad e=\frac{R}{d},
  \qquad \lambda=R.
\end{equation}

\section{Reciprocal-polar radius translation}

Reciprocal polarity provides a new representation of focal conics: instead of
describing a conic directly, one may represent it by the unique circle whose
reciprocal polar it is.  In this representation the action of the cone projection
becomes especially simple, merely increasing the circle radius by the constant
amount~$R$.  We now make this correspondence explicit.

Besant proves that the reciprocal polar of a circle with respect to an
auxiliary circle is a conic whose focus is the center of the auxiliary
circle and whose directrix is the polar of the center of the original
circle; he also gives its eccentricity and semi-latus rectum
\cite[Arts.~206--207, pp.~207--208]{Besant1890}.  In the present notation,
that classical theorem takes an especially transparent form.

\begin{definition}\label{def:pole}
For the circle $\Sigma_R$, the \emph{pole} of a line
\[
  a\cdot x=c,\qquad \norm a=1,\quad c\ne0,
\]
is the point $(R^2/c)a$.  The \emph{reciprocal polar} of a smooth curve is
the locus of the poles of its tangent lines; a tangent through $O$ has no finite
pole, and in the projective completion its pole is the ideal point in the
direction perpendicular to the tangent.
\end{definition}

\medskip
\noindent\emph{Notation.} For a smooth curve $\Gamma$, we write
$\mathcal{R}_{\Sigma_R}(\Gamma)$ for its reciprocal polar with respect to
$\Sigma_R$, that is, the locus of the poles (with respect to $\Sigma_R$) of
the tangent lines to $\Gamma$.

Figure~\ref{fig:rp-notation} illustrates this notation.  A tangent line
$t$ to $\Gamma$ is replaced by its pole $P$ with respect to $\Sigma_R$.
As the point of tangency varies along $\Gamma$, the corresponding poles
trace the reciprocal-polar curve $\mathcal{R}_{\Sigma_R}(\Gamma)$.

\begin{figure}[p]
\centering
\resizebox{\textwidth}{!}{%
\begin{tikzpicture}[
  scale=0.96,
  >={Latex[length=2.3mm]},
  line cap=round,
  line join=round,
  font=\small,
  original/.style={very thick,blue!60!black},
  reciprocal/.style={very thick,orange!85!black},
  tangent/.style={thick,gray!70},
  construction/.style={thin,dashed,gray!65},
  boundary/.style={densely dotted,thick,black!75},
  pole/.style={circle,fill=red!70!black,inner sep=1.55pt},
  maparrow/.style={->,very thick,red!70!black}
]

\def\R{1.00}
\def\g{0.80}
\def\q{1.40}

\begin{scope}[xshift=-8.6cm]
  \node[font=\bfseries] at (0,3.25) {1. Tangents to $\Gamma$};

  \draw[boundary] (0,0) circle (\R);
  \node[black!75,font=\footnotesize] at (-1.20,-1.30) {$\Sigma_R$};
  \fill (0,0) circle (1.5pt) node[below left] {$O$};

  \draw[original] (\g,0) circle (\q);
  \node[blue!60!black,fill=white,inner sep=1pt] at (1.78,1.45) {$\Gamma$};

  \foreach \th in {120,145,180,215,240}{
    \pgfmathsetmacro{\qx}{\g+\q*cos(\th)}
    \pgfmathsetmacro{\qy}{\q*sin(\th)}
    \pgfmathsetmacro{\tx}{-sin(\th)}
    \pgfmathsetmacro{\ty}{cos(\th)}
    \draw[tangent]
      ({\qx-1.20*\tx},{\qy-1.20*\ty}) --
      ({\qx+1.20*\tx},{\qy+1.20*\ty});
    \fill[blue!60!black] (\qx,\qy) circle (1.2pt);
  }

\end{scope}

\begin{scope}
  \node[font=\bfseries] at (0,3.25) {2. One tangent $t$ gives one pole $P$};

  \begin{scope}[scale=1.2]
  \def\th{150}
  \pgfmathsetmacro{\Qx}{\g+\q*cos(\th)}
  \pgfmathsetmacro{\Qy}{\q*sin(\th)}
  \pgfmathsetmacro{\ux}{cos(\th)}
  \pgfmathsetmacro{\uy}{sin(\th)}
  \pgfmathsetmacro{\tx}{-sin(\th)}
  \pgfmathsetmacro{\ty}{cos(\th)}
  \pgfmathsetmacro{\cval}{\q+\g*cos(\th)}
  \pgfmathsetmacro{\Prad}{\R*\R/\cval}
  \pgfmathsetmacro{\Px}{\Prad*\ux}
  \pgfmathsetmacro{\Py}{\Prad*\uy}
  \pgfmathsetmacro{\Hx}{\cval*\ux}
  \pgfmathsetmacro{\Hy}{\cval*\uy}
  \pgfmathsetmacro{\half}{sqrt(\R*\R-\cval*\cval)}

  \draw[boundary] (0,0) circle (\R);
  \node[black!75,font=\footnotesize] at (-1.00,-1.02) {$\Sigma_R$};

  \draw[original] (\g,0) circle (\q);
  \node[blue!60!black,fill=white,inner sep=1pt] at (1.62,1.30) {$\Gamma$};

  \draw[tangent]
    ({\Qx-1.00*\tx},{\Qy-1.00*\ty}) --
    ({\Qx+1.65*\tx},{\Qy+1.65*\ty});
  \node[gray!70!black,fill=white,inner sep=0.8pt]
    at ({\Qx-0.85*\tx},{\Qy-0.85*\ty}) {$t$};

  \draw[black!55,line width=1.0pt]
    ({\Hx+\half*\tx},{\Hy+\half*\ty}) -- ({\Hx-\half*\tx},{\Hy-\half*\ty});
  \fill[black!80] ({\Hx+\half*\tx},{\Hy+\half*\ty}) circle (1.1pt);
  \fill[black!80] ({\Hx-\half*\tx},{\Hy-\half*\ty}) circle (1.1pt);
  \node[black!80,font=\footnotesize,fill=white,inner sep=0.8pt]
    at ({\Hx+\half*\tx-0.22},{\Hy+\half*\ty+0.02}) {$A$};
  \node[black!80,font=\footnotesize,fill=white,inner sep=0.8pt]
    at ({\Hx-\half*\tx-0.02},{\Hy-\half*\ty+0.24}) {$B$};

  \fill[blue!60!black] (\Qx,\Qy) circle (1.3pt);
  \node[blue!60!black,fill=white,inner sep=0.8pt] at ({\Qx-0.26},{\Qy+0.03}) {$Q$};

  \draw[construction] (0,0)--(\Hx,\Hy);
  \fill[gray!55!black] (\Hx,\Hy) circle (1.2pt);
  \node[gray!45!black,font=\footnotesize,fill=white,inner sep=0.8pt]
    at ({\Hx-0.02},{\Hy-0.30}) {$H$};
  \draw[gray!60] ($(\Hx,\Hy)+(-0.10*\tx,-0.10*\ty)$)--
                 ($(\Hx,\Hy)+(-0.10*\tx+0.10*\ux,-0.10*\ty+0.10*\uy)$)--
                 ($(\Hx,\Hy)+(0.10*\ux,0.10*\uy)$);

  \draw[construction,->] (0,0)--({1.12*\Px},{1.12*\Py});
  \node[pole] (P) at (\Px,\Py) {};
  \node[red!70!black,fill=white,inner sep=0.8pt] at ({\Px-0.08},{\Py+0.24}) {$P$};
  \node[black!75,font=\scriptsize,fill=white,inner sep=0.6pt]
    at (-0.98,0.10) {$OH\cdot OP=R^2$};

  \fill (0,0) circle (1.35pt) node[below right,inner sep=1.5pt] {$O$};
  \end{scope}

\end{scope}

\begin{scope}[xshift=8.6cm]
  \node[font=\bfseries] at (0,3.25) {3. Sample tangents and their pole locus};

  \draw[boundary] (0,0) circle (\R);
  \node[black!75,font=\footnotesize] at (-1.20,-1.30) {$\Sigma_R$};
  \fill (0,0) circle (1.5pt) node[below left] {$O$};

  \draw[blue!60!black,opacity=0.35,thin]
    ([shift={(-85:\q)}]\g,0) arc[start angle=-85,end angle=85,radius=\q];
  \node[blue!60!black,opacity=0.75,font=\footnotesize] at (2.02,1.24) {$\Gamma$};

  \draw[reciprocal]
    plot[domain=0:360,samples=360,variable=\t]
    ({(\R*\R/(\q+\g*cos(\t)))*cos(\t)},
     {(\R*\R/(\q+\g*cos(\t)))*sin(\t)});

  \foreach \th in {-70,-40,-15,15,40,70}{
    \pgfmathsetmacro{\qx}{\g+\q*cos(\th)}
    \pgfmathsetmacro{\qy}{\q*sin(\th)}
    \pgfmathsetmacro{\tx}{-sin(\th)}
    \pgfmathsetmacro{\ty}{cos(\th)}
    \pgfmathsetmacro{\den}{\q+\g*cos(\th)}
    \pgfmathsetmacro{\rr}{\R*\R/\den}
    \pgfmathsetmacro{\px}{\rr*cos(\th)}
    \pgfmathsetmacro{\py}{\rr*sin(\th)}

    \draw[tangent,opacity=0.58]
      ({\qx-0.78*\tx},{\qy-0.78*\ty}) --
      ({\qx+0.78*\tx},{\qy+0.78*\ty});
    \fill[blue!60!black] (\qx,\qy) circle (1.1pt);
    \draw[construction,opacity=0.65] (0,0)--(\px,\py);
    \node[pole] at (\px,\py) {};
  }

  \foreach \th/\i in {70/1,-15/2,-70/3}{
    \pgfmathsetmacro{\qx}{\g+\q*cos(\th)}
    \pgfmathsetmacro{\qy}{\q*sin(\th)}
    \pgfmathsetmacro{\den}{\q+\g*cos(\th)}
    \pgfmathsetmacro{\rr}{\R*\R/\den}
    \node[gray!40!black,font=\small,fill=white,inner sep=0.6pt]
      at ({\qx+0.42*cos(\th)},{\qy+0.42*sin(\th)}) {$t_{\i}$};
    \node[red!70!black,font=\small,fill=white,inner sep=0.6pt]
      at ({(\rr+0.34)*cos(\th)},{(\rr+0.34)*sin(\th)}) {$P_{\i}$};
  }

  \node[orange!85!black,align=center,fill=white,inner sep=1.5pt] at (1.05,-2.18)
    {$\mathcal{R}_{\Sigma_R}(\Gamma)$\\locus of all poles};
  \draw[->,orange!85!black,thin] (0.72,-1.88)--(0.30,-1.10);
\end{scope}

\draw[maparrow] (-4.75,0.05)--(-4.15,0.05)
  node[midway,above,font=\footnotesize,black] {$t\mapsto P$};
\draw[maparrow] (4.15,0.05)--(4.75,0.05)
  node[midway,above,font=\footnotesize,black] {vary $t$};

\node[align=center,font=\small] at (0,-3.35)
  {$\displaystyle
    \mathcal{R}_{\Sigma_R}(\Gamma)
    =\{\text{poles, with respect to }\Sigma_R,\text{ of the tangent lines to }\Gamma\}. $};

\end{tikzpicture}
}
\caption{The reciprocal-polar notation $\mathcal{R}_{\Sigma_R}(\Gamma)$.
Panel~(1) shows the curve $\Gamma$, the reference circle $\Sigma_R$, and a
selection of its tangent lines.  In panel~(2), the pole $P$ of a tangent $t$ is
the inverse in $\Sigma_R$ of the foot $H$ of the perpendicular from $O$ to $t$;
in the configuration shown $\{A,B\}=t\cap\Sigma_R$ and $H$ is the midpoint of the
chord $AB$.  In panel~(3), each sample tangent $t_i$ yields a pole $P_i$; as the
point of
tangency varies along $\Gamma$, the poles trace the reciprocal-polar curve
$\mathcal{R}_{\Sigma_R}(\Gamma)$.}
\label{fig:rp-notation}
\end{figure}

Write
\[
  S^1=\{u\in\RR^2:\norm u=1\},
\]
the unit circle of oriented directions; by an \emph{interval of directions} we
mean a connected arc of $S^1$, possibly $S^1$ itself.  Let
\[
  C(G,q)=\{G+q u:u\in S^1\},\qquad G\ne O,\quad q>0,
\]
so that $G$ is the center of the circle and $q$ its radius; the unit vector $u$
is the outward normal, and $Q=G+qu$ is the point of tangency in the direction
$u$.  Write $G=g a$, where $g=\norm G$ and $a$ is a unit vector of angle
$\alpha$.  Since the tangent at $Q$ is perpendicular to the radius $GQ=qu$, its
outward unit normal is $u$; as $\norm u=1$, the tangent has the form $u\cdot x=c$,
where $c$ is the signed distance from $O$ to the tangent in the direction $u$.
Because the line passes through $Q$,
\[
  c=u\cdot Q=u\cdot(G+qu)=G\cdot u+q\,(u\cdot u)=q+G\cdot u,
  \qquad u\cdot u=\norm u^2=1 .
\]
The tangent at $Q$ is therefore
\[
  u\cdot x=q+G\cdot u=u\cdot Q .
\]
Its pole with respect to $\Sigma_R$ is therefore
\begin{equation}\label{eq:Pq}
  P_q(u)=\frac{R^2}{q+G\cdot u}\,u.
\end{equation}
For an interval of directions $I$ we write
$\left.\RP(C(G,q))\right|_I:=\{P_q(u):u\in I\}$ for the corresponding sub-arc of
the reciprocal-polar conic.
Whenever the denominator is positive, this is an ordinary positive focal
radius.  Geometrically the denominator $u\cdot Q$ is the signed distance from
$O$ to the tangent line along its unit normal $u$; its magnitude is the length
$OH$ of the perpendicular from $O$ to that line
(Figure~\ref{fig:signed-distance}).  It is positive exactly when the tangent
lies on the $u$-side of $O$, zero when the tangent passes through $O$, and
negative on the opposite side.
If $u=(\cos\theta,\sin\theta)$, then
\begin{equation}\label{eq:rp-polar}
  \rho_q(\theta)
  =\frac{R^2}{q+g\cos(\theta-\alpha)}
  =\frac{R^2/q}{1+(g/q)\cos(\theta-\alpha)}.
\end{equation}
Consequently the reciprocal-polar conic has
\begin{equation}\label{eq:rp-data}
  \lambda=\frac{R^2}{q},
  \qquad
  e=\frac{g}{q},
  \qquad
  \delta=\frac{\lambda}{e}=\frac{R^2}{g},
  \qquad
  \ell_G=\{x:G\cdot x=R^2\}.
\end{equation}
In particular, its type is read directly from the incidence of $O$ with the
circle $C(G,q)$:
\begin{equation}\label{eq:incidence-trichotomy}
  q<g\Longleftrightarrow\text{hyperbola},\qquad
  q=g\Longleftrightarrow\text{parabola},\qquad
  q>g\Longleftrightarrow\text{ellipse}.
\end{equation}

\smallskip
\noindent\emph{Inverse construction.}
The data~\eqref{eq:rp-data} invert, so the correspondence is a bijection: every
focal conic with focus $O$, eccentricity $e$, and semi-latus rectum $\lambda$ is
the reciprocal polar of a \emph{unique} circle $C(G,q)$.  Writing the directrix
in the oriented form
\[
  \ell=\{x:a\cdot x=\delta\},\qquad \norm a=1,\quad \delta=\frac{\lambda}{e}>0,
\]
the circle is recovered by
\[
  q=\frac{R^2}{\lambda},
  \qquad
  G=\frac{R^2}{\delta}\,a,
\]
so that $g=\norm G=eR^2/\lambda$.  Since the directrix is $a\cdot x=\delta$, its
pole with respect to $\Sigma_R$ is exactly $G$ (Definition~\ref{def:pole}); one
may therefore pass freely between a focal conic and its associated circle
$C(G,q)$.  The limiting case $q\to 0^{+}$, which recovers the Self-Directrix
Theorem, is treated in Corollary~\ref{cor:self-directrix}.

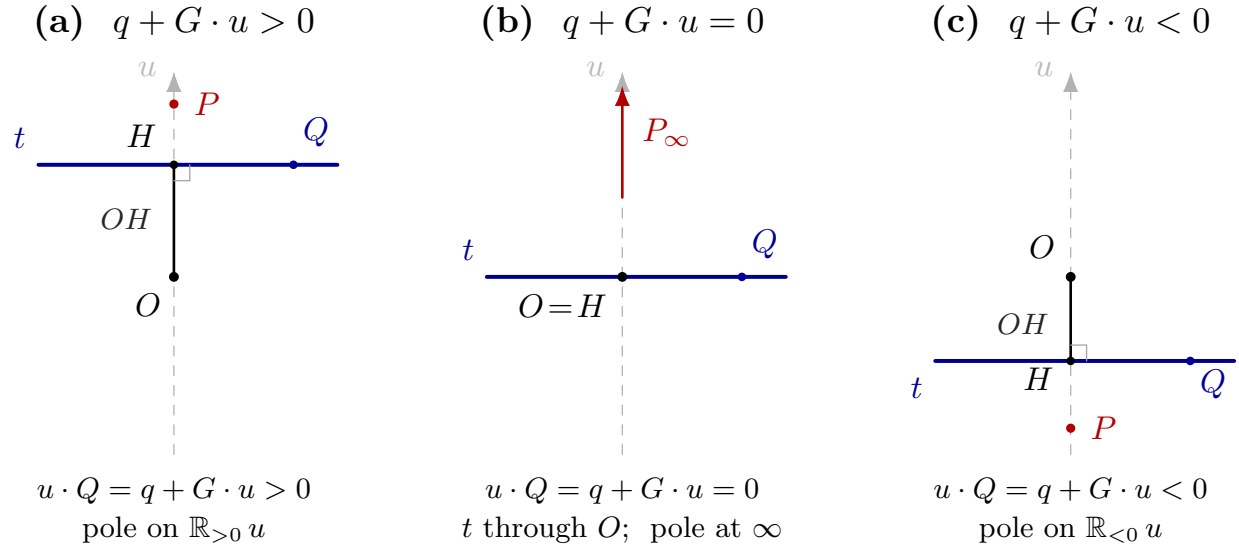
\begin{figure}[htbp]
\centering
\resizebox{\textwidth}{!}{%
\begin{tikzpicture}[
  >={Latex[length=2.2mm]},line cap=round,line join=round,font=\small,
  tang/.style={very thick,blue!55!black},
  axisline/.style={->,dashed,gray!60},
  perp/.style={thick,black},
  ra/.style={gray!70}
]

\begin{scope}[xshift=-4.8cm]
  \node[font=\bfseries] at (0,2.72) {(a)\ \ $q+G\cdot u>0$};
  \draw[axisline] (0,-1.9)--(0,2.2);
  \node[gray!55,anchor=south east] at (-0.05,2.02) {$u$};
  \draw[tang] (-1.45,1.2)--(1.75,1.2);
  \node[blue!55!black,anchor=south east] at (-1.45,1.24) {$t$};
  \draw[perp] (0,0)--(0,1.2);
  \node[gray!30!black,anchor=east,font=\footnotesize] at (-0.1,0.62) {$OH$};
  \fill (0,1.2) circle (1.3pt);
  \node[anchor=south east] at (-0.05,1.26) {$H$};
  \draw[ra] (0.17,1.2)--(0.17,1.03)--(0,1.03);
  \fill (0,0) circle (1.5pt);
  \node[anchor=north east] at (0.0,-0.04) {$O$};
  \fill[blue!60!black] (1.28,1.2) circle (1.3pt);
  \node[blue!60!black,anchor=south west] at (1.25,1.26) {$Q$};
  \fill[red!70!black] (0,1.85) circle (1.4pt);
  \node[red!70!black,anchor=west] at (0.08,1.85) {$P$};
  \node[align=center,font=\footnotesize] at (0,-2.5)
    {$u\cdot Q=q+G\cdot u>0$\\[1pt]pole on $\RR_{>0}\,u$};
\end{scope}

\begin{scope}
  \node[font=\bfseries] at (0,2.72) {(b)\ \ $q+G\cdot u=0$};
  \draw[axisline] (0,-1.9)--(0,2.2);
  \node[gray!55,anchor=south east] at (-0.05,2.02) {$u$};
  \draw[tang] (-1.45,0)--(1.75,0);
  \node[blue!55!black,anchor=south east] at (-1.45,0.05) {$t$};
  \fill (0,0) circle (1.5pt);
  \node[anchor=north east] at (-0.03,-0.05) {$O\!=\!H$};
  \fill[blue!60!black] (1.28,0) circle (1.3pt);
  \node[blue!60!black,anchor=south west] at (1.25,0.06) {$Q$};
  \draw[->,red!70!black,thick] (0,0.85)--(0,2.05);
  \node[red!70!black,anchor=west] at (0.08,1.55) {$P_\infty$};
  \node[align=center,font=\footnotesize] at (0,-2.5)
    {$u\cdot Q=q+G\cdot u=0$\\[1pt]$t$ through $O$;\ \ pole at $\infty$};
\end{scope}

\begin{scope}[xshift=4.8cm]
  \node[font=\bfseries] at (0,2.72) {(c)\ \ $q+G\cdot u<0$};
  \draw[axisline] (0,-1.9)--(0,2.2);
  \node[gray!55,anchor=south east] at (-0.05,2.02) {$u$};
  \fill (0,0) circle (1.5pt);
  \node[anchor=south east] at (-0.03,0.07) {$O$};
  \draw[tang] (-1.45,-0.9)--(1.75,-0.9);
  \node[blue!55!black,anchor=north east] at (-1.45,-0.94) {$t$};
  \draw[perp] (0,0)--(0,-0.9);
  \node[gray!30!black,anchor=east,font=\footnotesize] at (-0.1,-0.5) {$OH$};
  \fill (0,-0.9) circle (1.3pt);
  \node[anchor=north east] at (-0.05,-0.84) {$H$};
  \draw[ra] (0.17,-0.9)--(0.17,-0.73)--(0,-0.73);
  \fill[blue!60!black] (1.28,-0.9) circle (1.3pt);
  \node[blue!60!black,anchor=north west] at (1.25,-0.86) {$Q$};
  \fill[red!70!black] (0,-1.62) circle (1.4pt);
  \node[red!70!black,anchor=west] at (0.08,-1.62) {$P$};
  \node[align=center,font=\footnotesize] at (0,-2.5)
    {$u\cdot Q=q+G\cdot u<0$\\[1pt]pole on $\RR_{<0}\,u$};
\end{scope}

\end{tikzpicture}
}
\caption{Geometric meaning of the sign of the denominator $q+G\cdot u$ in the
pole $P_q(u)=R^2u/(q+G\cdot u)$.  Since $\norm u=1$ and $Q=G+qu$, the number
$q+G\cdot u=u\cdot Q$ is the signed distance from $O$ to the tangent along its
unit normal $u$, of magnitude $OH$.  Its sign places the pole on the positive
ray $\RR_{>0}u$~(a), at infinity~(b), or on the opposite ray $\RR_{<0}u$~(c); the
interval $I$ of Theorem~\ref{thm:rp-translation} selects the positive case~(a),
the focus-side branch.}
\label{fig:signed-distance}
\end{figure}

\FloatBarrier
\begin{theorem}[Reciprocal-polar radius translation]\label{thm:rp-translation}
Let $I$ be any interval of directions on which
$q+G\cdot u>0$; equivalently, the tangent to $C(G,q)$ lies on the $u$-side of
$O$ and its pole lies on the positive ray $\RR_{>0}u$.  Then, using the same
direction parameter $u$,
\begin{equation}\label{eq:rp-translation}
  f_R\bigl(P_q(u)\bigr)=P_{q+R}(u),\qquad u\in I.
\end{equation}
Equivalently,
\begin{equation}\label{eq:rp-curve-translation}
  f_R\bigl(\RP(C(G,q))\big|_I\bigr)
  =\RP(C(G,q+R))\big|_I.
\end{equation}
Thus reciprocal polarity linearizes the Confocal-Codirectrix action:
the center $G$ stays fixed and the reciprocal-circle radius is translated by
exactly $R$,
\[
  \boxed{\ q\longmapsto q+R\ }.
\]
\end{theorem}

\begin{proof}
Equation~\eqref{eq:Pq} gives
\[
  \frac{1}{\norm{P_q(u)}}=\frac{q+G\cdot u}{R^2}.
\]
Applying the lens identity~\eqref{eq:lens},
\[
  \frac{1}{\norm{f_R(P_q(u))}}
  =\frac{q+G\cdot u}{R^2}+\frac{1}{R}
  =\frac{q+R+G\cdot u}{R^2}.
\]
The right-hand side is $1/\norm{P_{q+R}(u)}$.  Since both $f_R(P_q(u))$ and
$P_{q+R}(u)$ lie on the positive ray $\RR_{>0}u$, they coincide:
$f_R(P_q(u))=P_{q+R}(u)$.
\end{proof}

\begin{figure}[htbp]
\centering
\begin{tikzpicture}[
  scale=0.93,
  >={Latex[length=2.3mm]},
  line cap=round,line join=round,font=\small,
  conic/.style={very thick},
  helper/.style={thick,dashed,gray!65},
  boundary/.style={densely dotted,thick,black!75},
  pull/.style={->,red!70!black,thin}
]
\def\Rr{1.5}
\def\gg{2.1}
\def\qq{1.4}
\def\qp{2.9}
\pgfmathsetmacro{\dd}{\Rr*\Rr/\gg}
\pgfmathsetmacro{\lamA}{\Rr*\Rr/\qq}
\pgfmathsetmacro{\eccA}{\gg/\qq}
\pgfmathsetmacro{\lamB}{\Rr*\Rr/\qp}
\pgfmathsetmacro{\eccB}{\gg/\qp}
\pgfmathsetmacro{\thstar}{acos(-1/\eccA)}

\begin{scope}[xshift=-6.4cm]
  \node[font=\bfseries] at (2.1,3.45) {reciprocal-circle side};
  \draw[->,darkgray] (-1.65,0)--(5.35,0);
  \draw[boundary] (0,0) circle (\Rr);
  \node[black!75,font=\footnotesize] at (-1.35,-1.30) {$\Sigma_R$};
  \draw[helper] (\dd,-3.0)--(\dd,3.0);
  \node[gray!70!black,font=\footnotesize,anchor=south east] at (\dd-0.06,2.58)
    {$\ell_G$};
  \node[gray!70!black,font=\scriptsize,anchor=north east] at (\dd-0.06,2.54)
    {polar of $G$};
  \draw[conic,blue!60!black] (\gg,0) circle (\qq);
  \draw[conic,orange!85!black] (\gg,0) circle (\qp);
  \fill (0,0) circle (1.5pt) node[below left] {$O$};
  \fill (\gg,0) circle (1.5pt) node[below right] {$G$};
  \node[blue!60!black,fill=white,inner sep=1pt] at (3.75,1.20) {$C(G,q)$};
  \node[orange!85!black,fill=white,inner sep=1pt] at (4.78,-2.12) {$C(G,q+R)$};
  \draw[->,very thick,gray!70] (2.1,2.05) arc[start angle=90,end angle=45,radius=0.8];
  \node[fill=white,inner sep=1.5pt] at (2.95,2.55) {$q\mapsto q+R$};
\end{scope}

\node[align=center,font=\footnotesize] at (0,0.15)
  {$\xleftrightarrow{\text{reciprocal polarity}}$};

\begin{scope}[xshift=6.5cm]
  \node[font=\bfseries] at (0,3.45) {focal-conic side};
  \clip (-3.6,-3.05) rectangle (3.05,3.1);
  \draw[->,darkgray] (-3.2,0)--(3.0,0);
  \draw[boundary] (0,0) circle (\Rr);
  \draw[helper] (\dd,-3.1)--(\dd,3.1);
  \node[gray!70!black,font=\footnotesize,anchor=west] at (1.16,1.72) {$\ell_G$};

  \draw[conic,blue!60!black]
    plot[domain=-128:128,samples=360,variable=\t]
    ({\lamA/(1+\eccA*cos(\t))*cos(\t)},
     {\lamA/(1+\eccA*cos(\t))*sin(\t)});

  \draw[conic,orange!85!black]
    plot[domain=-\thstar:\thstar,samples=360,variable=\t]
    ({\lamB/(1+\eccB*cos(\t))*cos(\t)},
     {\lamB/(1+\eccB*cos(\t))*sin(\t)});

  \foreach \t in {-95,-55,-20,20,55,95}{
    \pgfmathsetmacro{\ro}{\lamA/(1+\eccA*cos(\t))}
    \pgfmathsetmacro{\ri}{\lamB/(1+\eccB*cos(\t))}
    \draw[pull] ({\ro*cos(\t)},{\ro*sin(\t)}) --
                ({\ri*cos(\t)},{\ri*sin(\t)});
  }

  \pgfmathsetmacro{\xe}{\Rr*cos(\thstar)}
  \pgfmathsetmacro{\ye}{\Rr*sin(\thstar)}
  \draw[orange!85!black,fill=white] (\xe,\ye) circle (1.7pt);
  \draw[orange!85!black,fill=white] (\xe,-\ye) circle (1.7pt);

  \fill (0,0) circle (1.5pt) node[below] {$O$};
  \node[black!75,font=\footnotesize] at (-1.75,-1.15) {$\partial D_R$};
  \node[blue!60!black,align=center,font=\scriptsize] at (-2.25,2.60)
    {$\RP(C(G,q))$\\[-1pt]focus-side branch};
  \draw[orange!85!black,thin] (0.22,-1.85)--(-0.08,-0.78);
  \node[orange!85!black,align=center,font=\scriptsize,fill=white,inner sep=1pt,anchor=north]
    at (0.36,-1.88) {$\RP(C(G,q+R))\big|_I$\\[-1pt]ellipse arc};
\end{scope}
\end{tikzpicture}
\caption{Reciprocal-polar linearization.  Left: the center $G$ is fixed while
the reciprocal-circle radius increases from $q$ to $q+R$; in the parameters
drawn, the expanding circle passes across $O$.  Right: reciprocal polarity
turns the two circles into conics with the same focus $O$ and the same
directrix $\ell_G$.  The blue focus-side hyperbola branch is pulled radially inward by
$f_R$ (red arrows) to the orange arc of the reciprocal-polar ellipse.  The
orange endpoints are approached but not attained and lie on
$\partial D_R$.}
\label{fig:rp-translation}
\end{figure}
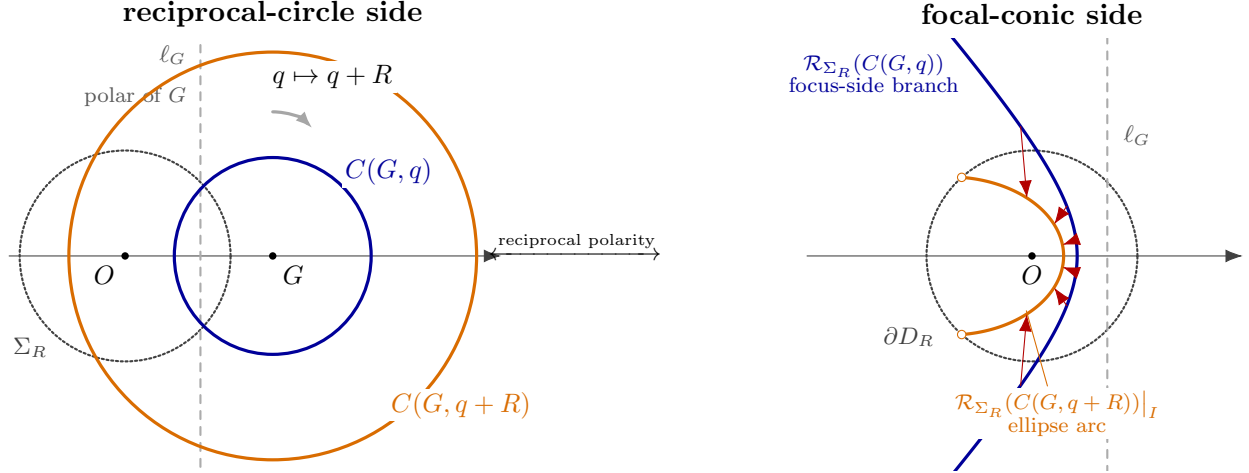

The parameter identities in~\eqref{eq:cc-law} now become immediate.  Indeed,
from~\eqref{eq:rp-data},
\[
  \lambda'=\frac{R^2}{q+R},
  \qquad
  e'=\frac{g}{q+R},
\]
so, using $\delta=R^2/g$ from~\eqref{eq:rp-data},
\[
  \frac1{\lambda'}=\frac{q+R}{R^2}=\frac1\lambda+\frac1R,
  \qquad
  \frac1{e'}=\frac{q+R}{g}=\frac1e+\frac{R}{g}=\frac1e+\frac{\delta}{R},
\]
and $\lambda'/e'=R^2/g=\delta$.
The hyperbola-parabola-ellipse transition becomes the elementary event of
the expanding circle $C(G,q)$ passing through $O$.

\begin{corollary}[The Self-Directrix Theorem as the limiting case $q\to0^{+}$]\label{cor:self-directrix}
Let $\ell=\{x\cdot n=d\}$, with $d>0$ and $\norm n=1$, and let
\[
  G=\frac{R^2}{d}\,n,
  \qquad
  I_0=\{u\in S^1:G\cdot u>0\}.
\]
Then, with $P_0(u):=R^2u/(G\cdot u)$ the limit of $P_q(u)$ as $q\to 0^{+}$,
\begin{equation}\label{eq:self-rp}
  \ell=\{P_0(u):u\in I_0\}
  \quad\xrightarrow{\ f_R\ }\quad
  \{P_R(u):u\in I_0\}
  =\left.\RP(C(G,R))\right|_{I_0},
\end{equation}
and the carrier of the image arc has focus $O$, directrix $\ell$, eccentricity
$R/d$, and semi-latus rectum $R$.
\end{corollary}

\begin{proof}
Here $G$ is the pole of $\ell$ with respect to $\Sigma_R$, and the corollary is
the degenerate member of the reciprocal-polar correspondence, obtained as the
circle $C(G,q)$ shrinks to its center.  For each $u\in I_0$, formula~\eqref{eq:Pq}
gives $P_q(u)\to P_0(u)$ as $q\to 0^{+}$.  Since $G\cdot P_0(u)=R^2$, every point
$P_0(u)$ lies on $\ell$; conversely, if $x\in\ell$ then $u=x/\norm{x}\in I_0$ and
\[
  P_0(u)=\frac{R^2}{G\cdot(x/\norm{x})}\,\frac{x}{\norm{x}}=x,
\]
so $u\mapsto P_0(u)$ parameterizes the whole line $\ell$.  (As $C(G,0)=\{G\}$ is a
point with no tangent family, $\RP(C(G,0))$ is understood as this limiting
reciprocal polar.)  Applying Theorem~\ref{thm:rp-translation} in the limit
$q\to 0^{+}$ yields~\eqref{eq:self-rp}, the open arc between the latus-rectum
endpoints.  By~\eqref{eq:rp-data} its carrier has
\[
  e=\frac{\norm G}{R}=\frac{R}{d},
  \qquad
  \lambda=\frac{R^2}{R}=R,
\]
the data of the Self-Directrix conic~\eqref{eq:self-directrix-data}; and, since
$G$ is the pole of $\ell$, its directrix is $\ell_G=\{x:G\cdot x=R^2\}=\ell$.
Thus $f_R$ carries $\ell$ to a conic with focus $O$ and directrix $\ell$ itself,
the self-directrix property that names the theorem.  Its type follows from
$e=R/d$: a hyperbola, parabola, or ellipse according as $d<R$, $d=R$, or $d>R$,
i.e.\ as $\ell$ meets, touches, or misses $\Sigma_R$.
\end{proof}

\medskip
\noindent
In the reciprocal-polar representation the Self-Directrix Theorem is simply the
endpoint of the radius-translation law,
\[
  \boxed{\,C(G,0)\ \xrightarrow{\ q\mapsto q+R\ }\ C(G,R)\,}
\]
and is therefore not an isolated phenomenon but the limiting member of the
general reciprocal-polar translation law.

\FloatBarrier

\section{The parabolic envelope of normals}

Besant records the following exercise: for a system of conics with a common
focus and a common latus rectum, the normals at the intersections with a
fixed line through the focus envelope a parabola; its focus lies on the
common latus-rectum line, and the fixed line is tangent at the vertex
\cite[Chap.~IX, Ex.~26, p.~189]{Besant1890}.  The self-directrix family
allows this envelope to be written explicitly.

Normalize the parallel input lines as
\[
  \ell_d=\{x=d\},\qquad d>0.
\]
By the Self-Directrix Theorem, let $\Gamma_d$ denote the carrier of
$f_R(\ell_d)$.  Its focus-side branch, which contains the actual image arc, has
polar equation
\begin{equation}\label{eq:family-polar}
  r=\frac{R}{1+(R/d)\cos\theta}.
\end{equation}
Every $\Gamma_d$ has focus $O$, axis the $x$-axis, and semi-latus rectum
$R$; hence the line $x=0$ is the common latus-rectum line.

Fix an angle
\begin{equation}\label{eq:phi-domain}
  -\frac{\pi}{2}<\phi<\frac{\pi}{2},\qquad \phi\ne0,
\end{equation}
put
\[
  u=(\cos\phi,\sin\phi),\qquad v=(-\sin\phi,\cos\phi),
\]
and let
\[
  L_\phi=\RR u,\qquad L_\phi^{+}=\{ru:r>0\}
\]
be the line through $O$ in direction $u$ and its positive ray.  In rotated
coordinates $(\xi,\eta)$, with $x=\xi u+\eta v$, the line $L_\phi$ is $\eta=0$.

\begin{theorem}[Normal-envelope parabola]\label{thm:normal-envelope}
Let $P_d=L_\phi^{+}\cap\Gamma_d$ (the intersection in direction $u$, i.e.\ the
polar angle $\theta=\phi$), and let $N_d$ be the normal to $\Gamma_d$ at $P_d$.  As $d$ ranges over $(0,\infty)$, the lines $N_d$ are tangent to the
fixed carrier parabola
\begin{equation}\label{eq:normal-envelope}
  \boxed{\qquad
  \eta=\frac{\tan\phi}{4R}(R-\xi)^2.
  \qquad}
\end{equation}
The points of tangency trace precisely the open arc $-R<\xi<R$ of this
parabola.  Its vertex is
\begin{equation}\label{eq:V}
  V=Ru\in\partial D_R,
\end{equation}
and $L_\phi$ is its tangent at $V$.  Its focus is
\begin{equation}\label{eq:J}
  J=Ru+R\cot\phi\,v
   =\left(0,\frac{R}{\sin\phi}\right),
\end{equation}
which lies on the common latus-rectum line $x=0$.
\end{theorem}

\begin{proof}
At the fixed angle $\theta=\phi$, equation~\eqref{eq:family-polar} gives
\begin{equation}\label{eq:rd}
  P_d=r_d u,
  \qquad
  r_d=\frac{Rd}{d+R\cos\phi}.
\end{equation}
The restriction $\cos\phi>0$ also shows that $P_d$ lies on the actual
self-directrix image arc: the ray in direction $u$ meets $\ell_d$ at
$z_d=(d,d\tan\phi)$, and $f_R(z_d)=\dfrac{Rd}{d+R\cos\phi}\,u=P_d$.
On its focus-side branch, and in particular near $P_d$, the conic $\Gamma_d$ has
the implicit equation
\begin{equation}\label{eq:implicit-Gd}
  d\sqrt{x^2+y^2}+Rx-Rd=0.
\end{equation}
At $P_d=r_d u$, a normal direction is therefore
\[
  d u+R(1,0).
\]
Since $(1,0)=\cos\phi\,u-\sin\phi\,v$, its components in the
$(u,v)$ basis are
\[
  \bigl(d+R\cos\phi,\,-R\sin\phi\bigr).
\]
Hence the normal line through $(r_d,0)$ is
\begin{equation}\label{eq:normal-d}
  R\sin\phi\,(\xi-r_d)+(d+R\cos\phi)\eta=0.
\end{equation}
As $d$ varies from $0$ to $\infty$, $r_d$ varies from $0$ to $R$.  Solving
\eqref{eq:rd} for $d$ gives
\[
  d=\frac{Rr\cos\phi}{R-r},
\]
where $r=r_d$.  Substitution in~\eqref{eq:normal-d} reduces the one-parameter
family of normals to
\begin{equation}\label{eq:normal-r}
  F(r;\xi,\eta)
  :=\sin\phi\,(R-r)(\xi-r)+R\cos\phi\,\eta=0,
  \qquad 0<r<R.
\end{equation}
The envelope is determined by
$F=0$ and $\partial F/\partial r=0$.  Since
\[
  \frac{\partial F}{\partial r}
  =\sin\phi\,(2r-R-\xi),
\]
we obtain $r=(R+\xi)/2$.  Substituting this value into
\eqref{eq:normal-r} yields
\[
  -\frac{\sin\phi}{4}(R-\xi)^2+R\cos\phi\,\eta=0,
\]
which is~\eqref{eq:normal-envelope}.

Writing~\eqref{eq:normal-envelope} as
\[
  (\xi-R)^2=4R\cot\phi\,\eta
\]
shows that its vertex is $(R,0)$, its tangent there is $\eta=0$, and its
focus in rotated coordinates is $(R,R\cot\phi)$.  Converting back to
$(x,y)$ gives~\eqref{eq:J}.
\end{proof}

The point of contact of $N_d$ with the envelope is also explicit.  If $r=r_d$, then
\begin{equation}\label{eq:contact-normal}
  \xi_d=2r-R,
  \qquad
  \eta_d=\frac{\tan\phi}{R}(R-r)^2,
  \qquad -R<\xi_d<R.
\end{equation}
The endpoints $\xi=-R$ and $\xi=R$ are limiting points corresponding to
$d\to 0^{+}$ and $d\to\infty$, respectively; in particular, the fixed
line $L_\phi$ is the limiting normal at the vertex $V$.

\begin{figure}[!t]
\centering
\begin{tikzpicture}[
  scale=1.05,
  >={Latex[length=2.3mm]},
  line cap=round,line join=round,font=\small,
  conic/.style={very thick,blue!60!black},
  normal/.style={thick,red!70!black},
  envelope/.style={very thick,orange!85!black},
  envelope extension/.style={very thick,orange!85!black,densely dashed},
  boundary/.style={densely dotted,thick,black!75},
  helper/.style={thick,dashed,gray!65}
]
\def\RRR{2.0}
\def\ph{35}
\pgfmathsetmacro{\cc}{cos(\ph)}
\pgfmathsetmacro{\ss}{sin(\ph)}
\pgfmathsetmacro{\ttan}{tan(\ph)}
\pgfmathsetmacro{\JY}{\RRR/\ss}

\clip (-4.7,-3.75) rectangle (4.45,3.95);
\draw[->,darkgray] (-4.25,0)--(4.35,0) node[below right]{common axis};
\draw[helper] (0,-2.05)--(0,3.92);
\node[gray!70!black,font=\footnotesize,anchor=west] at (0.5,3.62)
  {common latus-rectum line};
\draw[->,gray!65,thin] (0.45,3.68)--(0.03,3.85);
\draw[boundary] (0,0) circle (\RRR);
\node[black!75,font=\footnotesize] at (-1.55,-1.58) {$\partial D_R$};

\draw[conic,opacity=0.58]
  plot[domain=-104:104,samples=300,variable=\t]
  ({\RRR/(1+(\RRR/0.6)*cos(\t))*cos(\t)},
   {\RRR/(1+(\RRR/0.6)*cos(\t))*sin(\t)});
\draw[conic,opacity=0.76]
  plot[domain=-138:138,samples=320,variable=\t]
  ({\RRR/(1+cos(\t))*cos(\t)},
   {\RRR/(1+cos(\t))*sin(\t)});
\draw[conic]
  plot[domain=0:360,samples=320,variable=\t]
  ({\RRR/(1+0.5*cos(\t))*cos(\t)},
   {\RRR/(1+0.5*cos(\t))*sin(\t)});

\node[blue!60!black,font=\footnotesize,fill=white,inner sep=0.8pt] at (-0.78,3.32) {$d<R$};
\node[blue!60!black,font=\footnotesize,fill=white,inner sep=0.8pt] at (-2.10,3.34) {$d=R$};
\node[blue!60!black,font=\footnotesize,fill=white,inner sep=0.8pt] at (-3.35,-1.28) {$d>R$};

\draw[very thick,gray!70]
  (0,0)--({3.7*\cc},{3.7*\ss});
\node[gray!70!black,fill=white,inner sep=1pt] at ({3.35*\cc},{3.35*\ss+0.42}) {$L_\phi$};

\draw[envelope extension]
  plot[domain=-3.0:-2.0,samples=70,variable=\xi]
  ({\xi*\cc-(\ttan/(4*\RRR))*(\RRR-\xi)^2*\ss},
   {\xi*\ss+(\ttan/(4*\RRR))*(\RRR-\xi)^2*\cc});
\draw[envelope]
  plot[domain=-2.0:2.0,samples=220,variable=\xi]
  ({\xi*\cc-(\ttan/(4*\RRR))*(\RRR-\xi)^2*\ss},
   {\xi*\ss+(\ttan/(4*\RRR))*(\RRR-\xi)^2*\cc});
\draw[envelope extension]
  plot[domain=2.0:3.15,samples=70,variable=\xi]
  ({\xi*\cc-(\ttan/(4*\RRR))*(\RRR-\xi)^2*\ss},
   {\xi*\ss+(\ttan/(4*\RRR))*(\RRR-\xi)^2*\cc});
\node[orange!85!black,anchor=east,fill=white,inner sep=1pt] at (-2.55,0.86)
  {envelope arc};
\draw[->,orange!85!black,thin] (-2.48,0.78)--(-1.92,0.38);

\fill (0,0) circle (1.5pt) node[below right] {$O$};
\fill (0,\JY) circle (1.6pt) node[right] {$J$};
\pgfmathsetmacro{\VX}{\RRR*\cc}
\pgfmathsetmacro{\VY}{\RRR*\ss}
\draw[orange!85!black,fill=white] (\VX,\VY) circle (1.8pt);
\node[orange!85!black,fill=white,inner sep=1pt] at (1.45,1.72) {$V$};

\foreach \dval/\lab/\exta/\extb in {0.6/1/0.38/0.72,2.0/2/0.45/0.90,4.0/3/0.55/1.05}{
  \pgfmathsetmacro{\rr}{\RRR*\dval/(\dval+\RRR*\cc)}
  \pgfmathsetmacro{\px}{\rr*\cc}
  \pgfmathsetmacro{\py}{\rr*\ss}
  \pgfmathsetmacro{\xii}{2*\rr-\RRR}
  \pgfmathsetmacro{\etaa}{\ttan/\RRR*(\RRR-\rr)^2}
  \pgfmathsetmacro{\tx}{\xii*\cc-\etaa*\ss}
  \pgfmathsetmacro{\ty}{\xii*\ss+\etaa*\cc}
  \pgfmathsetmacro{\ax}{\tx-\exta*(\px-\tx)}
  \pgfmathsetmacro{\ay}{\ty-\exta*(\py-\ty)}
  \pgfmathsetmacro{\bx}{\px+\extb*(\px-\tx)}
  \pgfmathsetmacro{\by}{\py+\extb*(\py-\ty)}
  \draw[normal] (\ax,\ay)--(\bx,\by);
  \fill[blue!60!black] (\px,\py) circle (1.45pt);
  \fill[orange!85!black] (\tx,\ty) circle (1.25pt);
}
\node[red!70!black,font=\footnotesize,anchor=west] at (2.05,0.18) {selected normals};
\draw[->,red!70!black,thin] (2.00,0.24)--(1.34,0.57);
\end{tikzpicture}
\caption{The normal-envelope theorem for the self-directrix sweep.  Three
members of the common-focus, common-latus-rectum family are shown in blue:
the focus-side branch of a hyperbola ($d<R$), a parabola ($d=R$), and an
ellipse ($d>R$).  Their selected intersection points lie on the fixed positive
ray $L_\phi^{+}$.  The red normal at each selected point is tangent to the solid
orange envelope arc of the carrier parabola~\eqref{eq:normal-envelope}; the
dashed orange pieces show its continuation.  The fixed line $L_\phi$ is tangent
to the carrier parabola at the limiting vertex $V=Ru\in\partial D_R$, while the
focus $J$ lies on the common latus-rectum line.}
\label{fig:normal-envelope}
\end{figure}
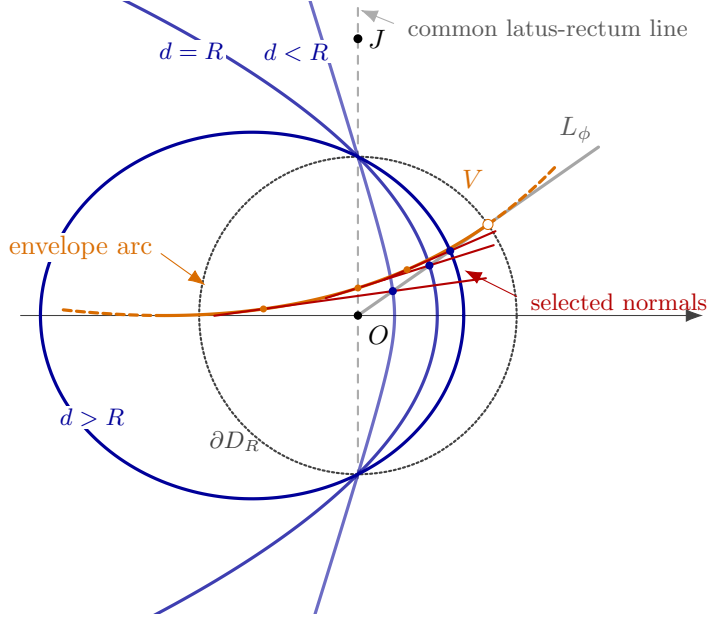

\section{A direct tangent construction from the input line}

A classical focus-directrix theorem states that the line from the focus to
the point where a tangent meets the directrix is perpendicular to the focal
radius through the point of contact
\cite[Prop.~III, p.~8]{Besant1890}.  Because the Self-Directrix Theorem uses
the \emph{input line itself} as the directrix of the image conic, this gives
a particularly economical construction.

\begin{proposition}[Tangent from the original point and line]\label{prop:tangent-construction}
Let $\ell$ be a line not through $O$, let $z\in\ell$, and put
\[
  w=f_R(z).
\]
Let $m=Oz$ be the line through $O$ and $z$, and let $m^\perp$ be the line
through $O$ perpendicular to $m$.  If
\begin{equation}\label{eq:F-construction}
  F=\ell\cap m^\perp
\end{equation}
is finite, then the tangent to the carrier conic of $f_R(\ell)$ at $w$ is
\begin{equation}\label{eq:tangent-Fw}
  \boxed{\quad Fw.\quad}
\end{equation}
If $m^\perp\parallel\ell$ (the near-vertex case), the tangent at $w$ is
parallel to $\ell$.
\end{proposition}

\begin{proof}
By radiality, $O,z,w$ are collinear.  By the Self-Directrix Theorem, the
carrier $\Gamma$ of $f_R(\ell)$ has focus $O$ and directrix $\ell$.
It therefore suffices to verify the classical tangent-directrix fact in the
present normalization.

Rotate coordinates so that $\ell=\{x=d\}$, $d>0$.  On the focus-side branch of
$\Gamma$, which contains the image point $w$, the carrier has equation
\begin{equation}\label{eq:tangent-implicit}
  r=e(d-x),\qquad r=\sqrt{x^2+y^2},
\end{equation}
with $e=R/d$.  Let $w=(x,y)$, with $y\ne0$, and let the tangent at $w$ meet
$\ell$ at $F=(d,\eta)$.  Differentiating
$r+ex-ed=0$ (with $\nabla r=(x/r,y/r)$) shows that a normal at $w$ is
\[
  \left(\frac{x}{r}+e,\frac{y}{r}\right).
\]
Hence the tangent equation, evaluated at $F$, is
\[
  (x+er)(d-x)+y(\eta-y)=0.
\]
Using $r=e(d-x)$ and
$r^2=x^2+y^2=e^2(d-x)^2$, the left-hand side reduces to
\[
  dx+\eta y.
\]
Thus $F\cdot w=0$, which is exactly $OF\perp Ow$.  The point
$\ell\cap (Ow)^\perp$ is unique, so it is the intersection of the tangent
with the directrix.  Since $Ow=Oz$, this point is precisely the $F$ in
\eqref{eq:F-construction}.  When $y=0$, the perpendicular through $O$ is
parallel to $\ell$, and the vertex tangent is parallel to the directrix.
\end{proof}

In normalized coordinates the construction is especially explicit.  If
\[
  \ell=\{x=d\},\qquad z=(d,t),\qquad t\ne0,
\]
then
\begin{equation}\label{eq:F-explicit}
  F=\left(d,-\frac{d^2}{t}\right),
\end{equation}
because $F\cdot z=0$.  One needs only the original line $\ell$, the original
ray $Oz$, and a perpendicular through $O$; no conic parameters need be
computed.

\begin{figure}[htbp]
\centering
\begin{tikzpicture}[
  scale=1.05,
  >={Latex[length=2.3mm]},
  line cap=round,line join=round,font=\small,
  conic/.style={very thick,orange!85!black},
  tangent/.style={very thick,blue!60!black},
  helper/.style={thick,dashed,gray!65},
  boundary/.style={densely dotted,thick,black!75},
  pull/.style={->,red!70!black,thick}
]
\def\Rz{1.5}
\def\dz{3.0}
\def\tz{4.5}
\pgfmathsetmacro{\rhoZ}{sqrt(\dz*\dz+\tz*\tz)}
\pgfmathsetmacro{\sfac}{\Rz/(\Rz+\rhoZ)}
\pgfmathsetmacro{\wx}{\sfac*\dz}
\pgfmathsetmacro{\wy}{\sfac*\tz}
\pgfmathsetmacro{\Fy}{-\dz*\dz/\tz}
\pgfmathsetmacro{\ee}{\Rz/\dz}

\draw[->,darkgray] (-3.55,0)--(3.8,0) node[below right]{axis};
\draw[boundary] (0,0) circle (\Rz);
\node[black!75,font=\footnotesize] at (-1.06,-1.34) {$\partial D_R$};
\draw[helper] (\dz,-2.4)--(\dz,4.7);
\node[gray!70!black,anchor=west] at (\dz+0.08,1.55) {directrix $\ell$};

\draw[conic]
  plot[domain=0:360,samples=320,variable=\t]
  ({\Rz/(1+\ee*cos(\t))*cos(\t)},
   {\Rz/(1+\ee*cos(\t))*sin(\t)});
\node[orange!85!black] at (-2.15,2.0) {carrier conic};

\draw[gray!60,thick] (-0.15*\dz,-0.15*\tz)--(1.03*\dz,1.03*\tz);
\draw[helper] (-0.48*\dz,0.48*\dz*\dz/\tz)--(1.15*\dz,-1.15*\dz*\dz/\tz);

\draw[tangent] (\dz,\Fy)--({\wx-0.75*(\dz-\wx)},{\wy-0.75*(\Fy-\wy)});
\draw[pull,shorten <=3pt,shorten >=3pt] (\dz,\tz)--(\wx,\wy);

\coordinate (O) at (0,0);
\coordinate (Z) at (\dz,\tz);
\coordinate (W) at (\wx,\wy);
\coordinate (F) at (\dz,\Fy);
\fill (O) circle (1.5pt) node[below left] {$O$};
\fill[black] (Z) circle (1.55pt) node[right] {$z$};
\fill[orange!85!black] (W) circle (1.75pt);
\draw[gray!55,thin] (1.30,1.12)--(\wx+0.10,\wy+0.05);
\node[orange!85!black,anchor=west] at (1.30,1.15) {$w=f_R(z)$};
\fill[blue!60!black] (F) circle (1.65pt) node[right] {$F$};
\pic[draw,gray!70,angle radius=7pt] {right angle=F--O--W};

\node[blue!60!black,rotate=-51] at (2.30,-0.8) {tangent $Fw$};
\node[gray!70!black,font=\footnotesize,rotate=-34] at (-1.45,0.92)
  {$(Oz)^\perp$};
\end{tikzpicture}
\caption{Direct construction of the tangent.  The point $z$ on the input
line $\ell$ is pulled radially inward to $w=f_R(z)$ on the self-directrix
conic.  Draw the line through $O$ perpendicular to the ray $Oz=Ow$; it meets
$\ell$ at $F$.  The blue line $Fw$ is the tangent to the carrier conic at
$w$.}
\label{fig:tangent-construction}
\end{figure}
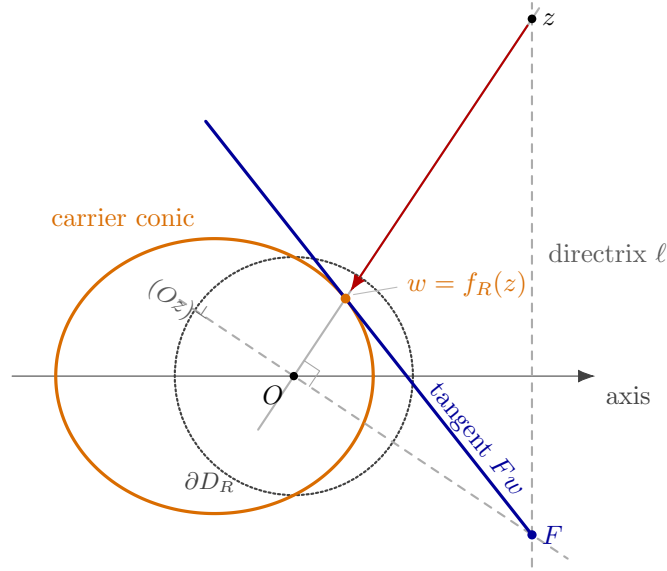

\FloatBarrier

\section{Concluding remarks}

The three results arise from two elementary features of the cone projection.
Under the bijective correspondence between focal conics and their associated
circles, the lens identity~\eqref{eq:lens} becomes the literal radius
translation $q\mapsto q+R$ of Theorem~\ref{thm:rp-translation}.  In the limiting line case it also yields the
common semi-latus rectum $R$ of the self-directrix sweep, whose normals along a
fixed ray envelope the explicit parabola of Theorem~\ref{thm:normal-envelope}.
The complementary ray-preserving property places $w=f_R(z)$ on the ray $Oz$,
which is exactly what the tangent construction of
Proposition~\ref{prop:tangent-construction} uses.  In each case a classical
theorem of Besant supplies the geometry and the cone projection supplies the
specific action, giving an elementary, synthetic complement to the analytic
development of~\cite{Georgiou-cone}.

\end{document}